\documentclass[11pt]{amsart}
\usepackage[foot]{amsaddr}
\usepackage[english]{babel}

\allowdisplaybreaks

\usepackage[
pdftex,hyperfootnotes]{hyperref}
\hypersetup{
	colorlinks=true,
	linkcolor=NavyBlue, 
	urlcolor=RoyalPurple,
	citecolor=OliveGreen,
	pdftitle={Banded totally positive matrices and mixed multiple orthogonal polynomials},
	bookmarks=true,
}
\usepackage[usenames,dvipsnames,svgnames,table,x11names]{xcolor}
\usepackage{pgfplots}
\usepackage{tikz}
\usepackage{tikz-3dplot}
\usetikzlibrary{automata,quotes, chains,matrix,calc,shadows,shapes.callouts,shapes.geometric,shapes.misc,positioning,patterns,decorations.shapes,
	decorations.pathmorphing,decorations.markings,decorations.fractals,decorations.pathreplacing,shadings,fadings,arrows.meta,bending}

\usepackage{nicematrix,drawmatrix}
\usepackage[utf8]{inputenc}

\usepackage{comment}
\usepackage[textwidth=17.5cm,textheight=22.275cm,
height=
24.275cm,
width=18cm]{geometry}

\usepackage{amssymb,latexsym,amsmath,amsthm,bm}
\usepackage{mathrsfs}
\usepackage{mathtools,arydshln,mathdots}

\usepackage{tcolorbox}

\usepackage[table]{xcolor}

\usepackage{Baskervaldx}
\usepackage[]{newtxmath}

\usepackage{bigints}

\theoremstyle{plain}

\newtheorem{teo}{Theorem}[section]
\newtheorem{coro}[teo]{Corollary}
\newtheorem{lemma}[teo]{Lemma}
\newtheorem{pro}[teo]{Proposition}
\newtheorem{defi}[teo]{Definition}
\newtheorem{rem}[teo]{Remark}

\renewcommand{\d}{\operatorname{d}}

\newcommand{\B}{\mathbb{B}}

\newcommand{\N}{\mathbb{N}}
\newcommand{\R}{\mathbb{R}}

\def\UT{\mbox{\begin{picture}(7,10)
		\put(6,0){\line(-1,2){5}}
			\put(6,0){\line(0,1){10}}
			\put(1.1,10){\line(1,0){5.1}}
		\end{picture}
}}

\def\LT{\mbox{\begin{picture}(7,10)
			\put(.9,0){\line(1,0){5}}
			\put(6,0){\line(-1,2){5}}
			\put(1,0){\line(0,1){10}}
		\end{picture}
}}

\allowdisplaybreaks[1]

\makeatletter
\DeclareRobustCommand{\gaussk}{\DOTSB\gaussk@\slimits@}
\newcommand{\gaussk@}{\mathop{\vphantom{\sum}\mathpalette\bigcal@{K}}}

\newcommand{\bigcal@}[2]{%
	\vcenter{\m@th
		\sbox\z@{$#1\sum$}%
		\dimen@=\dimexpr\ht\z@+\dp\z@
		\hbox{\resizebox{!}{0.8\dimen@}{$\mathcal{K}$}}%
	}%
}
\newcommand{\cfracplus}{\mathbin{\cfracplus@}}
\newcommand{\cfracplus@}{%
	\sbox\z@{$\dfrac{1}{1}$}%
	\sbox\tw@{$+$}%
	\raisebox{\dimexpr\dp\tw@-\dp\z@\relax}{$+$}%
}
\newcommand{\cfracdots}{\mathord{\cfracdots@}}
\newcommand{\cfracdots@}{%
	\sbox\z@{$\dfrac{1}{1}$}%
	\sbox\tw@{$+$}%
	\raisebox{\dimexpr\dp\tw@-\dp\z@\relax}{$\cdots$}%
}
\makeatother

\makeatletter
\newcommand*{\relrelbarsep}{.386ex}
\newcommand*{\relrelbar}{%
	\mathrel{%
		\mathpalette\@relrelbar\relrelbarsep
	}%
}
\newcommand*{\@relrelbar}[2]{%
	\raise#2\hbox to 0pt{$\m@th#1\relbar$\hss}%
	\lower#2\hbox{$\m@th#1\relbar$}%
}
\providecommand*{\rightrightarrowsfill@}{%
	\arrowfill@\relrelbar\relrelbar\rightrightarrows
}
\providecommand*{\leftleftarrowsfill@}{%
	\arrowfill@\leftleftarrows\relrelbar\relrelbar
}
\providecommand*{\xrightrightarrows}[2][]{%
	\ext@arrow 0359\rightrightarrowsfill@{#1}{#2}%
}
\providecommand*{\xleftleftarrows}[2][]{%
	\ext@arrow 3095\leftleftarrowsfill@{#1}{#2}%
}
\makeatother

\catcode`,\active

\catcode`\,12

\usepackage{appendix}

\usepackage{xcolor} 

\usepackage[normalem]{ulem}
\usepackage{soul} 

\usepackage{comment} 

\usepackage{tikz}
\usetikzlibrary{shapes,arrows}
\usepackage{verbatim}

\tikzstyle{block} = [draw, rectangle, 
minimum height=3em, minimum width=2em]

\usepackage{mathrsfs} 

\begin{document}
	\title[Banded totally positive matrices and multiple orthogonality]{Banded totally positive matrices and \\
	 normality for mixed  multiple orthogonal polynomials}
	\author[A Branquinho]{Amílcar Branquinho$^{1}$}
	\address{$^1$CMUC, Departamento de Matemática,
		Universidade de Coimbra, 3001-454 Coimbra, Portugal}
	\email{$^1$ajplb@mat.uc.pt}

	\author[A Foulquié]{Ana Foulquié-Moreno$^{2}$}
	\address{$^2$CIDMA, Departamento de Matemática, Universidade de Aveiro, 3810-193 Aveiro, Portugal}
	\email{$^2$foulquie@ua.pt}

	\author[M Mañas]{Manuel Mañas$^{3}$}
	\address{$^3$Departamento de Física Teórica, Universidad Complutense de Madrid, Plaza Ciencias 1, 28040-Madrid, Spain 
	}
	\email{$^3$manuel.manas@ucm.es}
	


\begin{abstract}
This paper serves as an introduction to banded totally positive matrices, exploring various characterizations and associated properties. A significant result within is the demonstration that the collection of such matrices forms a semigroup, notably including a subset permitting positive bidiagonal factorization. Moreover, the paper applies this concept to investigate step line  normality concerning the degrees of associated recursion polynomials. It presents a spectral Favard theorem, ensuring the existence of measures, thereby guaranteeing that these recursion polynomials represent mixed multiple orthogonal polynomials that maintain normality on the step line  indices.
\end{abstract}
	
	\maketitle

\section{Introduction}

In the papers \cite{BFM1,BFM2,BFM3}, we delve into the analysis of bounded banded semi-infinite matrices, termed as PBF matrices, which possess factorizations involving positive bidiagonal matrices. This investigation is conducted in connection with the theory of mixed multiple orthogonal polynomials. Within this context, we establish a Favard spectral theorem, ensuring the existence of measures for mixed multiple orthogonality, wherein the recurrence matrix aligns with the structure of the banded matrix. As a result, we derive explicit degrees of precision for a mixed multiple version of Gauss quadrature.

Within the theory of multiple orthogonal polynomials, cf. \cite{andrei_walter,nikishin_sorokin, Ismail,afm},  the concept of index normality holds significance, ensuring the attainment of maximum degrees for the orthogonal polynomials. A notable subset within this framework is the step line , representing the specific polynomials that emerge as recursion polynomials.

However, in \cite{BFM1,BFM2}, a mistaken claim regarding the presence of step line  normality was made. Although partially accurate, this error did not compromise the validity of other results, notably the spectral Favard theorem and the Gauss quadrature formulas.

To address the determination of step line  normality for PBF mixed multiple orthogonal polynomials, we initiate a study of a novel set of matrices: those that are banded and feature strictly positive nontrivial minors. While triangular totally positive matrices have been subject to prior investigation, analysis on such banded matrices remains largely unexplored but for \cite{Price,Metelmann}. We explore several properties of these matrices and apply them to investigate step line  normality for the corresponding mixed multiple orthogonal polynomials.
	
\subsection{Preliminaries on matrix analysis}
Let's establish some standard notation for matrix analysis. Consider an \( n \times n \) real matrix \( A \) and a sequence of \( r \) columns \( \boldsymbol{j} = \{j_1, \dots, j_r\} \) and a sequence of \(r\) rows  \( \boldsymbol{i} = \{i_1, \dots, i_r\} \), where both are increasing sequences of numbers in \( \{1, \dots, n\} \). We define the submatrix
\[
A[\boldsymbol{i}; \boldsymbol{j}] \equiv A\begin{bNiceMatrix}
	i_1 & \Cdots & i_r\\
	j_1 & \Cdots & j_r
\end{bNiceMatrix}
\]
where the entries are the elements found at the intersection of the corresponding rows and columns. Minors are the determinants of these submatrices:
\[
A(\boldsymbol{i}; \boldsymbol{j}) \equiv A\begin{pNiceMatrix}
	i_1 & \Cdots & i_r\\
	j_1 & \Cdots & j_r
\end{pNiceMatrix} \coloneq \det A\begin{bNiceMatrix}
	i_1 & \Cdots & i_r\\
	j_1 & \Cdots & j_r
\end{bNiceMatrix}.
\]
Submatrices and minors are termed proper if not all rows and columns are included.
	
Principal submatrices and minors are obtained by selecting the same set of column labels as row labels:
	\begin{align*}
	\begin{aligned}
	A[\boldsymbol i]	\coloneq	A[\boldsymbol i; \boldsymbol i]		&=A\begin{bNiceMatrix}
		i_1 &\Cdots & i_r\\
		i_1&\Cdots & i_r
	\end{bNiceMatrix}, & 	A(\boldsymbol i)	\coloneq	A(\boldsymbol i; \boldsymbol i)	&=A\begin{pNiceMatrix}
	i_1 &\Cdots & i_r\\
	i_1&\Cdots & i_r
	\end{pNiceMatrix}.
	\end{aligned}
	\end{align*} 
	
In the set \( \mathscr{Q}(r,n) \) of increasing sequences of \( r \) numbers within \( \{1,\dots,n\} \), the dispersion of an element \( \boldsymbol{i} \in \mathscr{Q}(r,n) \) is defined as 
\[
d(\boldsymbol{i}) = \sum_{k=1}^{r-1}(i_{k+1}-i_k-1) = i_r - i_1 - (r-1).
\]
Elements \( \boldsymbol{i} \in \mathscr{Q}(r,n) \) with zero dispersion, denoted by \( d(\boldsymbol{i}) = 0 \), are also referred to as contiguous because they take the form \( \boldsymbol{i} = \{i,i+1,\dots,i+r-1\} \). The set of increasing sequences of labels \( \mathscr{Q}(r,n) \) is partially ordered such that \( \boldsymbol{i} \leq \boldsymbol{j} \) if and only if \( i_k \leq j_k \) for \( k \in \{1,\dots,r\} \). The infimum element in this lattice is \( \{1,2,\dots,r\} \).

Contiguous submatrices and minors correspond to contiguous sequences (zero dispersion) and take the form:
\[
\begin{aligned}
	&A\begin{bNiceMatrix}
		i & i+1 & \Cdots & i+r-1\\
		j & j+1 & \Cdots & j+r-1
	\end{bNiceMatrix}, &A\begin{pNiceMatrix}
		i & i+1 & \Cdots & i+r-1\\
		j & j+1 & \Cdots & j+r-1
	\end{pNiceMatrix}.
\end{aligned}
\]
Particular cases of contiguous  submatrices and minors 
 those containing the first rows or columns associated with the infimum of the lattice of submatrix labels. These submatrices are known as initial submatrices.

The right and left shadows of a contiguous submatrix
\[
A\begin{bNiceMatrix}
	i & i+1 & \Cdots & i+r-1\\
	j & j+1 & \Cdots & j+r-1
\end{bNiceMatrix},
\]
of a matrix \( A \in \mathbb{R}^{n \times n} \), are given by the initial submatrices:
\[
\begin{aligned}
	&A\begin{bNiceMatrix}
		1 & \Cdots & i+r-1\\
		j & \Cdots & n
	\end{bNiceMatrix}, &A\begin{bNiceMatrix}
		i & \Cdots & n\\
		1 & \Cdots & j+r-1
	\end{bNiceMatrix}.
\end{aligned}
\]

If a matrix has all its minors nonnegative, it is termed totally nonnegative (TN); if the minors are positive (greater than zero), it is termed totally positive (TP). The sets of such matrices are denoted by TN and TP, respectively. The set of nonsingular matrices in TN is denoted by InTN. Recall that a matrix $M$ is reducible if and only if there exists a permutation matrix $P$ such that $PMP^\top = \begin{bNiceMatrix}[small]
	A & B \\
	0 & D
\end{bNiceMatrix}$, where $A$ and $D$ are square blocks and $B$ is the off-diagonal block in the upper right corner, while the lower left corner contains a zero block.  The set of irreducible nonsingular totally nonnegative matrices is denoted by IITN, cf. \cite{Fallat}, and is also known as oscillatory matrices. Oscillatory matrices $M$ are TN matrices such that for  some $n$ the matrix $M^n$ is TP.
For comprehensive discussions on this subject, refer to the books by Pinkus \cite{Pinkus book} and Fallat \cite{Fallat}, or consult Gantmacher's classical work \cite{Gantmacher}. We also recommend the nice and short review \cite{Fallat_0}.

A property we will later employ is the following classical statement, often referred to as Fischer's inequality because he first proved it for positive semidefinite matrices \cite[Theorem 7.8]{Horn}:

\begin{teo}[Fischer's inequality]\label{teo:Classical}
	For an \( n\times n \) TN matrix \( A \) and \( p<n \), it holds that
	\[
	A\begin{pNiceMatrix}
		1 & \Cdots & n\\
		1 & \Cdots & n
	\end{pNiceMatrix} \leq A\begin{pNiceMatrix}
		1 & \Cdots & p\\
		1 & \Cdots & p
	\end{pNiceMatrix} A\begin{pNiceMatrix}
		p+1 & \Cdots & n\\
		p+1 & \Cdots & n
	\end{pNiceMatrix}.
	\]
\end{teo}

Another significant result is Michel Fekete's criterion \cite{Fekete}.

\begin{teo}[Fekete, 1913]
	A matrix \( A \in \R^{n\times n} \) is TP if and only if \( A(\boldsymbol{i};\boldsymbol{j}) > 0 \) for all contiguous sequences; i.e., \( d(\boldsymbol{i}) = d(\boldsymbol{j}) = 0 \).
\end{teo}

An extension to TN matrices due to Colin Cryer was found in \cite{Cryer2}, see also \cite[Theorem 2.1]{Ando}, which, when the matrix is invertible, leads to the Fekete--Cryer criterion:

\begin{teo}[Cryer, 1976]\label{teo:Fekete-Ando}
	A nonsingular matrix \( A \in \R^{n\times n} \) is TN if and only if \( A(\boldsymbol{i};\boldsymbol{j}) \geq 0 \) for \( \boldsymbol{i},\boldsymbol{j} \in \mathscr{Q}(r,n) \) such that \( d(\boldsymbol{j}) = 0 \).
\end{teo}

In 1996, Mariano Gasca and Juan Manuel Peña \cite{Gasca-Pena} provided the following important characterization:

\begin{teo}[Gasca--Peña, 1996]
	Let \( A \in \R^{n\times n} \) be nonsingular.
	\begin{enumerate}
		\item \( A \) is InTN if and only if for each \( r \in \{1,2, \dots, n\} \):
		\begin{enumerate}
			\item Principal leading minors are positive, \( A(\{1,2, \dots, r\}) > 0 \), and
			\item Initial column minors are nonnegative, i.e., \( A(\boldsymbol{i} ;\{1,2, \ldots, r\}) \geq 0 \),  \( \boldsymbol{i} \in \mathscr{Q}(r,n) \), and
			\item Initial row minors are nonnegative, i.e., \( A(\{1,2, \ldots, r\}; \boldsymbol{j}) \geq 0 \),   \( \boldsymbol{j} \in \mathscr{Q}(r,n) \).
		\end{enumerate}
		\item \( A \) is TP if and only if for each \( r \in \{1,2, \dots, n\} \):
		\begin{enumerate}
			\item Initial column  minors are positive, i.e., \( A(\boldsymbol{i} ;\{1,2, \ldots, r\}) > 0 \), \( \boldsymbol{i} \in \mathscr{Q}(r,n) \), \( d(\boldsymbol{i}) = 0 \), and
			\item Initial row  minors are positive, i.e., \( A(\{1,2, \ldots, r\} ; \boldsymbol{j}) > 0 \),  \( \boldsymbol{j} \in \mathscr{Q}(r,n) \), \( d(\boldsymbol{j}) = 0 \).
		\end{enumerate}
	\end{enumerate}
\end{teo}

In a nonsingular totally nonnegative matrix, it is of interest to detect or characterize the singular submatrices. Allan Pinkus \cite{Pinkus 2008} proved a particularly interesting characterization of such minors. See also the treatment in \cite{Pinkus book}.

\begin{defi}[Pinkus' submatrices and minors]
	Given \( A \in \) InTN, a Pinkus submatrix is a contiguous submatrix
	\[
	A\begin{bNiceMatrix}
		\alpha & \alpha+1 & \Cdots & \alpha+r-1\\
		\beta& \beta+1 & \Cdots & \beta+r-1
	\end{bNiceMatrix},
	\]
	which is singular and does not contain a singular principal submatrix.
\end{defi}

\begin{teo}[Pinkus, 2008]\label{teo:Pinkus}
	Let \( A \in \) InTN, and let \( \{\alpha_k,\beta_k,r_k\}_{k=1}^l \) be the set of triples giving all the Pinkus' submatrices. Then,
	\begin{enumerate}
		\item \( \alpha_k \neq \beta_k \).
		\item \( A\begin{pNiceMatrix}[small]
			i_1 & \Cdots & i_r\\
			j_1 & \Cdots & j_r
		\end{pNiceMatrix} = 0 \) if and only if for some \( k \in \{1,\dots,l\} \), there exists a \( r_k \times r_k \) principal submatrix of \( A\begin{bNiceMatrix}[small]
			i_1 & \Cdots & i_r\\
			j_1 & \Cdots & j_r
		\end{bNiceMatrix} \) that belongs to the right shadow of the Pinkus submatrix \( A\begin{bNiceMatrix}[small]
			\alpha_k & \alpha_k+1 & \Cdots & \alpha_k+r_k-1\\
			\beta_k & \beta_k+1 & \Cdots & \beta_k+r_k-1
		\end{bNiceMatrix} \), when \( \alpha_k<\beta_k \), and to the corresponding left shadow for \( \alpha_k>\beta_k \).
	\end{enumerate}
\end{teo}

	\section{Banded totally positive matrices}
	
We are investigating a specific class of matrices known as banded matrices, which are characterized by having $p$ subdiagonals and $q$ superdiagonals. These matrices are denoted as B$_{p,q}$. Furthermore, we define B$_{p,q}$TN as the subset of B$_{p,q}$ that intersects with TN, B$_{p,q}$InTN as the subset of B$_{p,q}$ that intersects with InTN, and B$_{p,q}$IITN as the subset of B$_{p,q}$ that intersects with IITN. To streamline notation, we also use BTN, BInTN, and BIITN when it is unnecessary to specify the size of the band.

A particular case of banded matrices arises when $p=n$ and $q=0$, resulting in lower triangular matrices, and when $p=0$ and $q=n$, resulting in upper triangular matrices. These specific cases are denoted by 	\scalebox{.65}{\LT}  \hspace{-6pt}TN or 	\scalebox{.7}{\UT}  \hspace{-6pt}TN for lower and upper totally nonnegative matrices, respectively.

Harvey Price established the following theorem, as referenced in \cite{Price0} and \cite[Theorem 5.5]{Price}, with additional insights from \cite{Metelmann}:
\begin{teo}[Price, 1965]\label{teo:Price}
	Let $T \in \text{B}_{p,q}$ be a nonsingular banded matrix with contiguous minors
	\[
T\begin{pNiceMatrix}
	i&	i+1 &\Cdots & i+r-1\\
	j&	j+1&\Cdots & j+r-1
\end{pNiceMatrix}> 0,
	\]
	for $i-p \leq j \leq i+q$, where $i, j \in \{1, \dots, n\}$. Then, $T$ is a  oscillatory matrix.
\end{teo}

Now, let's introduce submatrices emerging from the banded structure, leading to the distribution of zeros outside the band and resulting in trivially vanishing minors. We refer to them as trivial submatrices or trivial minors.

\begin{defi}
	Trivial submatrices of $T \in$ B$_{p,q}$TN are those containing a zero originating from outside the band on their diagonal. In other words, they are submatrices of the form:
	\begin{align*}
		T\begin{bNiceMatrix}
			i_1 &\Cdots & i_r\\
			j_1&\Cdots & j_r
		\end{bNiceMatrix},
	\end{align*}
	where there exists a $k\in\{1,\dots,r\}$ such that $(i_k,j_k)$ lies strictly above the $q$-th superdiagonal or strictly below the $p$-th subdiagonal.
	
	Submatrices and minors that are not trivial are known as nontrivial submatrices and minors, respectively.
\end{defi}
	
Given the banded structure, if a submatrix is trivial, the zero in its diagonal will throw a corresponding shadow of zeros:

\begin{pro}
	The trivial submatrices
	\begin{align*}
		T\begin{bNiceMatrix}
			i_1 &\Cdots & i_r\\
			j_1&\Cdots & j_r
		\end{bNiceMatrix}
	\end{align*}
	with $T_{i_k,j_k}=0$ must satisfy the following conditions:
	\begin{enumerate}
		\item The zero entry at the diagonal $(i_k,j_k)$ throws a zero right shadow if this zero entry lies strictly above the $q$-th superdiagonal.
		\item The zero entry at the diagonal $(i_k,j_k)$ throws a zero left shadow if this zero entry lies strictly below the $p$-th subdiagonal.
	\end{enumerate}
\end{pro}

\begin{pro}
	Trivial submatrices of $T\in$ BTN are singular.
\end{pro}
\begin{proof}
	The trivial submatrices for $T_{i_k,j_k}=0$ above the $q$-th superdiagonal are of the form:
	\begin{align*}
		T\begin{bNiceMatrix}
			i_1 &\Cdots & i_r\\
			j_1&\Cdots & j_r
		\end{bNiceMatrix}=\begin{bNiceMatrix}
			*&\Cdots&&&*&0&\Cdots&&0\\\\
			\Vdots &&\Ddots& &\Vdots  &		\Vdots & &&\Vdots\\\\
			*&\Cdots&&&*&0&\Cdots&&0\\
			*&\Cdots&&&*&0&\Cdots&&0\\
			*&\Cdots&&&*&	*&*&\Cdots&*&\\
			\Vdots &&& &\Vdots  &		\Vdots & \Vdots&\Ddots&\Vdots\\
			*&\Cdots&&&*&	*&*&\Cdots&*&
		\end{bNiceMatrix},
	\end{align*}
	and the corresponding minor is obviously zero (consider a block computation of the determinant, and note that one of the diagonal blocks has a zero row). A similar argument applies for $T_{i_k,j_k}=0$ below the $p$-th subdiagonal and the corresponding left shadow of zeros:
	\begin{align*}
		T\begin{bNiceMatrix}
			i_1 &\Cdots & i_r\\
			j_1&\Cdots & j_r
		\end{bNiceMatrix}=\begin{bNiceMatrix}
			*&\Cdots&&&*&*&\Cdots&&*\\\\
			\Vdots &&\Ddots& &\Vdots  &		\Vdots & &&\Vdots\\\\
			*&\Cdots&&&*&*&\Cdots&&*\\
			0&\Cdots&&&0&0&*&\Cdots&*\\
			&&&&&	&*&\Cdots&*&\\
			\Vdots &&& &\Vdots  &		\Vdots & \Vdots&\Ddots&\Vdots\\
			0&\Cdots&&&0&	0&*&\Cdots&*&
		\end{bNiceMatrix}.
	\end{align*}
\end{proof}
	
Given $s\in\{1,\dots,n\}$, we consider the translation $\tau_s:\mathscr Q(r,n)\to \mathscr Q(r,n)$ defined as $\tau_s \boldsymbol{i}=\{i_1+s,\dots,i_r+s\}$, which is well-defined for $i_r+s\leq n$. Let $k$ be the largest index $k$ such that $i_{k}+s\leq n$, implying $i_{k+1}+s> n$; then, we define $\tau_s \boldsymbol{i}=\{i_1+s,\dots,i_{k}+s,n,\dots, n\}$.

\begin{pro}
	Nontrivial submatrices are of the form $T[\boldsymbol{i};\boldsymbol{j}]$ with $\boldsymbol{j} \leq \tau_q\boldsymbol{i}$ and $\boldsymbol{i}\leq \tau_p \boldsymbol{j}$.
\end{pro}

We now define the set of banded totally positive matrices:

\begin{defi}
	A  banded matrix is said to be  banded totally positive  if all nontrivial minors are positive. The set of such matrices will be denoted by BTP or by B$_{p,q}$TP.
\end{defi}

A special case of banded totally positive matrices is represented by upper triangular matrices with $p=0$ and $q=n$, and lower triangular matrices with $p=n$ and $q=0$. These sets are denoted by \scalebox{.7}{\UT}  \hspace{-6pt}TP or \scalebox{.65}{\LT}  \hspace{-6pt}TP, respectively.

	\begin{pro}[Nontrivial contiguous submatrices]
		For $T\in\R^{n\times n}$ such that $T\in$B$_{(p,q)}$TN, the corresponding nontrivial contiguous submatrices, associated with zero dispersion sequences of indexes,
		\begin{align*}
			T\begin{bNiceMatrix}
				i&	i+1 &\Cdots & i+r-1\\
				j&	j+1&\Cdots & j+r-1
			\end{bNiceMatrix},	
		\end{align*} 
		are those satisfying $i-p\leq j\leq i+q$, with $i,j\in\{1,\dots,n\}$.
	\end{pro}
	
	\begin{proof}
		Trivial contiguous submatrices are those whose diagonal does not lie within the band of $T$. If the diagonal lies in the band, they do not have a zero on their diagonal.
	\end{proof}

	Let's illustrate this with a diagram of a banded matrix. Trivial contiguous submatrices $T_2$ and $T_1$ are highlighted in red. Nontrivial contiguous submatrices $T_3$ and $T_4$, marked in green, lie precisely on the border of the band, where the submatrix's diagonal aligns with an extreme diagonal. Finally, nontrivial submatrices $T_5$ and $T_6$ are highlighted in blue, with their diagonals inside the band.
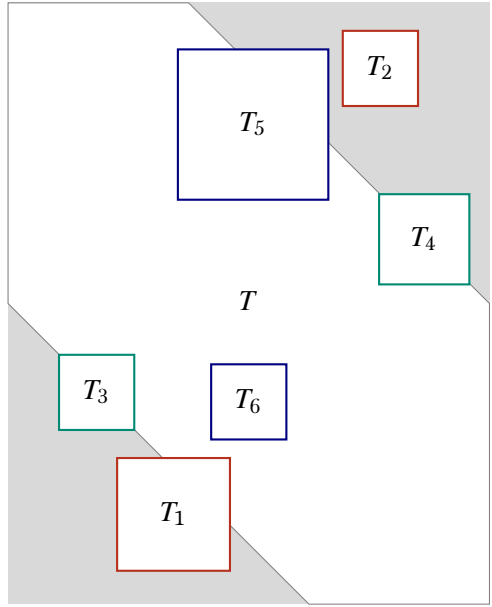
\begin{center}
	\drawmatrixset{bbox style={fill=Gray!30}}
$\drawmatrix[bbox/.append, ,scale=8,banded, width=.8]T$
	
\vspace*{-2cm}\hspace{-2cm}$\drawmatrix[scale=1.5,draw=BrickRed,thick]{T_1}$ 

\vspace*{-7cm}\hspace{0cm}
$\drawmatrix[scale=2,draw=NavyBlue,thick]{T_5}$

\vspace*{2cm}\hspace{-4.16cm}
$\drawmatrix[, scale=1,draw=PineGreen,thick]{T_3}$

\vspace*{-3.2cm}\hspace{4.55cm}
$\drawmatrix[, scale=1.2,draw=PineGreen,thick]{T_4}$

\vspace*{1cm}\hspace{0cm}$\drawmatrix[scale=1,draw=NavyBlue,thick]{T_6}$ 

\vspace*{-5,5cm}\hspace*{3.5cm}	$\drawmatrix[scale=1,draw=BrickRed,thick]{T_2}$ 
\vspace*{8cm}
\end{center}
		
We recast Price's Theorem \ref{teo:Price} as follows:
\begin{teo}\label{teo:Prince2} 
Banded matrices with  all nontrivial contiguous minors being positive  are oscillatory.
\end{teo}		
		
Inspired by the characterization of \scalebox{.65}{\LT}  \hspace{-6pt}TN or 	\scalebox{.7}{\UT}  \hspace{-6pt}TN \cite{Cryer}, we now present a preliminary characterization of banded totally positive matrices in terms contiguous minors:

\begin{teo}\label{teo:BTP vs contiguous minors}
	A matrix is banded totally positive if and only if all nontrivial contiguous minors are positive.
\end{teo}

\begin{proof}
In terms of necessity, we note that every nontrivial minor maintains is positive, including those formed by contiguous minors.

Regarding sufficiency, the application of Price's Theorem \ref{teo:Prince2} indicates the matrix's property of being IITN, thereby enabling the application of Pinkus' Theorem \ref{teo:Pinkus}. However, it's crucial to note that the trivial contiguous minors with $r>1$ often encompass a zero column or row, making them unsuitable for Pinkus' contiguous submatrix criteria. Consequently, Pinkus' contiguous submatrices are confined to $1\times 1$ submatrices, representing individual entries located above the $q$-th superdiagonal or below the $p$-th subdiagonal. Furthermore, it's noteworthy that the shadow thrown  by the Pinkus entries are consistently filled with zeros.
	
	Pinkus' Theorem \ref{teo:Pinkus} asserts that for the submatrix
	\begin{align*}
		T\begin{bNiceMatrix}
			i_1 &\Cdots & i_r\\
			j_1&\Cdots & j_r
		\end{bNiceMatrix}
	\end{align*}
	to be singular, it is necessary for it to contain a $r_k\times r_k$ principal submatrix lying in the corresponding shadow. Here, $r_k=1$, implying that the theorem demands a diagonal entry (since the minor must be principal) to reside in the corresponding shadow. However, the shadow is filled with zeros, indicating that the diagonal entry must be zero. Thus, the minor is trivial.
	
	We conclude that the zero minors are indeed the trivial ones, as asserted.
\end{proof}

Following \cite[Theorem 3.1]{Cryer} and \cite[page 85]{Karlin}, we provide an improved characterization. The idea is to examine a reduced family of zero-dispersion nontrivial submatrices, specifically initial submatrices and minors.
To achieve this objective, we will utilize the main result presented in \cite{Metelmann}. To facilitate this, let us introduce:
\begin{defi}
	For any matrix \(T\in\mathbb{R}^{n\times n}\), we define the $\Delta$ initial  submatrices as follows:
	\begin{align*}
		\begin{aligned}
			\Delta^i[T]&\coloneq T\begin{bNiceMatrix}
				1      & \Cdots & i     \\
				n-i+1 & \Cdots & n
			\end{bNiceMatrix}, & \Delta_i[T]&\coloneq T\begin{bNiceMatrix}
				n-i+	1 & \Cdots & n \\
				1      & \Cdots & i
			\end{bNiceMatrix},
		\end{aligned}
	\end{align*}
	and corresponding minors $\Delta^i(T)$ and $\Delta_i(T)$.
\end{defi}

In a banded matrix, several of these $\Delta$ minors trivially evaluate to zero. We refer to these as the trivial \(\Delta\)s.

\begin{lemma}
	If \(T\in\text{B}_{p,q}\), then 
	\begin{align*}
		\begin{aligned}
			\Delta^i(T)&=0, & \text{for } i\in\{1,\dots, n-q-1\},
		\end{aligned} &&
		\begin{aligned}
			\Delta_i(T)&=0, & \text{for } i\in\{1,\dots, n-p-1\},
		\end{aligned}
	\end{align*} 
	Hence, there are \(p+q-2\) nontrivial \(\Delta\)s. Namely, \(\Delta^i(T)\) for \(i\in\{n-q,\dots, n\}\) and \(\Delta_i(T)\) for \(i\in\{n-p,\dots, n\}\).
\end{lemma}

\begin{rem}
	Note that the submatrices $	\Delta^i[T]$ and $\Delta_i[T]$ correspond to matrices $M^o_{n-i+1}$ and $M^u_{n-i+1}$ in \cite{Metelmann}, respectively.  
\end{rem}

%

Subsequently, we can invoke Kurt Metelmann's equivalences, cf. \cite{Metelmann}:
\begin{teo}[Metelmann, 1973]\label{teo:Metelmann}
	Given a banded matrix $T$, the following equivalences hold:
	\begin{enumerate}
		\item All nontrivial submatrices $\Delta_i[T], \Delta^i[T]$ are InTN.
		\item All nontrivial initial  minors 
		\begin{align*}
		\begin{aligned}
				&T\begin{pNiceMatrix}
				1 & 2 & \Cdots & r \\
				j & j+1 & \Cdots & j+r-1
			\end{pNiceMatrix}, &
			&T\begin{pNiceMatrix}
				i & i+1 & \Cdots & i+r-1 \\
				1 & 2 & \Cdots & r
			\end{pNiceMatrix},
		\end{aligned}
		\end{align*}
		are positive.
		\item There exists a nonnegative bidiagonal factorization of $T$ into
		\(T = \hat{L}_1 \cdots \hat{L}_p \cdot D \cdot \hat{U}_q \cdots \hat{U}_1 \)
		where
		\(\hat{L}_i \coloneq \begin{bNiceMatrix}[small]
			I_{p-i} & 0 \\
			0 & L_i
		\end{bNiceMatrix} \)
		is a nonnegative bidiagonal lower unitriangular matrix with $L_i$ being a positive bidiagonal lower unitriangular matrix, and
		\(\hat{U}_i \coloneq \begin{bNiceMatrix}[small]
			I_{q-i} & 0 \\
			0 & U_i
		\end{bNiceMatrix} \)
		is a nonnegative bidiagonal upper unitriangular matrix with $U_i$ being a positive bidiagonal upper unitriangular matrix, and $D$ is a positive diagonal matrix.
	\end{enumerate}
\end{teo}

\begin{coro}\label{coro:Metelmann}
	If the banded matrix $T$ satisfies any of the equivalent conditions i) or ii) in Metelmann's Theorem \ref{teo:Metelmann}, then it is InTN.
\end{coro}

\begin{proof}
	It follows from the equivalent factorization property iii) that $T$, being a product of InTN matrices, is itself InTN.
\end{proof}

\begin{teo}
	A matrix $T$ is banded totally positive if and only if all
	 nontrivial initial  minors are positive.
\end{teo}

\begin{proof}
	For necessity, we note that initial minors are contiguous minors and therefore are positive.
	
	For sufficiency, we proceed as follows.   Since, according to Corollary \ref{coro:Metelmann}, the matrix $T$ is  InTN, the application of Fischer's inequality (cf. Theorem \ref{teo:Classical}) will lead to the desired conclusion.
	
We choose $j$ and $i$ such that the contiguous submatrices are nontrivial.		Consider the inequality:
	\begin{align}\label{eq:leq1}
		T\begin{pNiceMatrix}
			1 &\Cdots & i+r-1\\
			j-i+1&\Cdots & j+r-1
		\end{pNiceMatrix}\leq T\begin{pNiceMatrix}
			1 &\Cdots & i-1\\
			j-i+1&\Cdots & j-1
		\end{pNiceMatrix}T\begin{pNiceMatrix}
			i&i+1 &\Cdots & i+r-1\\
			j&j+1&\Cdots & j+r-1
		\end{pNiceMatrix},
	\end{align}
	which implies that if all nontrivial initial minors $T\begin{pNiceMatrix}[small]
		1&	2 &\Cdots & r\\
		j&	j+1&\Cdots & j+r-1
	\end{pNiceMatrix}$ do not vanish, then Equation \eqref{eq:leq1} ensures that all nontrivial contiguous submatrices above the main diagonal are nonsingular.
	
	Similarly, from Fischer's inequality we obtain:
	\begin{align}\label{eq:leq2}
		T\begin{pNiceMatrix}
			i-j+1 &\Cdots & i+r-1\\
			1&\Cdots & j+r-1
		\end{pNiceMatrix}\leq T\begin{pNiceMatrix}
			i&i+1 &\Cdots & i+r\\
			j&j+1&\Cdots & j+r
		\end{pNiceMatrix}T\begin{pNiceMatrix}
			i-j+1 &\Cdots & i-1\\
			1&\Cdots & j-1
		\end{pNiceMatrix},
	\end{align}
	which ensures that if all nontrivial initial  $T\begin{pNiceMatrix}[small]
		i&	i+1&\Cdots & i+r-1\\
		1&	2 &\Cdots & r
	\end{pNiceMatrix}$ are nonsingular, then all nontrivial contiguous submatrices above the main diagonal are nonsingular.
	
	Finally,  Theorem \ref{teo:BTP vs contiguous minors} leads to the result.
\end{proof}
%
		 
In Ando's seminal work \cite{Ando}, the text highlights a significant insight: \emph{"A criterion for total positivity of a band matrix is found in Metelmann (1973)"}, referencing \cite{Metelmann}. It's worth noting that Ando employs the term "total positive" where we typically use "total nonnegative," while what Ando terms ``strict total positive'' aligns with our understanding of "total positivity." However, as our analysis demonstrates, Metelmann's theorem effectively characterizes  banded totally positive matrices, or in Ando's terminology, banded  strictly totally nonnegative matrices.

\begin{teo}\label{teo:Mettelmann_modificdo}
	Given a banded matrix $T$, the following equivalences hold:
	\begin{enumerate}
		\item $T$ is a banded totally positive matrix.
		\item All nontrivial submatrices $\Delta_i[T], \Delta^i[T]$ are InTN.
		\item All nontrivial initial  minors 
		\begin{align*}
		\begin{aligned}
				&T\begin{pNiceMatrix}
				1 & 2 & \Cdots & r \\
				j & j+1 & \Cdots & j+r-1
			\end{pNiceMatrix}, &
			&T\begin{pNiceMatrix}
				i & i+1 & \Cdots & i+r-1 \\
				1 & 2 & \Cdots & r
			\end{pNiceMatrix},
		\end{aligned}
		\end{align*}
		are positive.
		\item There exists a nonnegative bidiagonal factorization of $T$ into
		\(T = \hat{L}_1 \cdots \hat{L}_p \cdot D \cdot \hat{U}_q \cdots \hat{U}_1 \)
		where
		\(\hat{L}_i \coloneq \begin{bNiceMatrix}[small]
			I_{p-i} & 0 \\
			0 & L_i
		\end{bNiceMatrix} \)
		is a nonnegative bidiagonal lower unitriangular matrix with $L_i$ being a positive bidiagonal lower unitriangular matrix, and
		\(\hat{U}_i \coloneq \begin{bNiceMatrix}[small]
			I_{q-i} & 0 \\
			0 & U_i
		\end{bNiceMatrix} \)
		is a nonnegative bidiagonal upper unitriangular matrix with $U_i$ being a positive bidiagonal upper unitriangular matrix, and $D$ is a positive diagonal matrix.
	\end{enumerate}
\end{teo}

We aim to refine the characterization of banded positive matrices, drawing inspiration
 from Boris and Michael Shapiro's paper \cite{Shapiro}, with a significant reduction in the number of minors to verify to just \(p+q+2\).
  Importantly, the number of non-zero minors to verify depends solely on the width of the band, rather than the size of the matrix. 
 
 It's worth noting that in \cite{Shapiro}, trivial submatrices are identified as those with no row or column below the main diagonal. However, there exist other trivial submatrices not fitting into this category. For instance, consider an upper triangular matrix \(T\) with \(n\geq 4\), which presents the trivial submatrix:
\[
T	\begin{bNiceMatrix}
	1 &3 & 4\\
	1 &2 &4
\end{bNiceMatrix}=\begin{bNiceMatrix}
	* &* &*\\
	0 &0 &*\\
	0 &0 &*
\end{bNiceMatrix}
\]
This submatrix, is trivially singular, but lacks a row or column below the main diagonal.

\begin{teo}\label{teo:Deltas}
	A matrix \(T\in\)InTN is BTP if and only if 
	all nontrivial \(\Delta\) initial  minors are nonzero.
\end{teo}

\begin{proof}
	For the necessity part, it's notable that if all nontrivial minors do not vanish, then the nontrivial \(\Delta\)'s do not vanish either.
	
	To demonstrate sufficiency , we present two proofs. The first one relies on Metelmann's Theorem \ref{teo:Mettelmann_modificdo}. Since the matrix is InTN, all its submatrices inherit this property, ensuring that the initial submatrices, denoted as $\Delta$, are TN. However, they are also assumed to be nonsingular. This implies that all these submatrices belong to InTN. Therefore, according to Theorem \ref{teo:Mettelmann_modificdo}, the matrix is deemed to be BTP.

The second proof draws inspiration from \cite{Shapiro}. Given the hypothesis that the matrix is InTN, we leverage Fisher's inequalities. Our approach initiates by assuming the nonzero status of all  \(\Delta\) nontrivial initial minors. Subsequently, we delve into the analysis of nontrivial contiguous minors, as expounded in Theorem \ref{teo:BTP vs contiguous minors}. Our aim is to demonstrate that if all nontrivial \(\Delta\) minors are nonzero, then the same holds true for all nontrivial contiguous minors.
	
	We pick a nontrivial contiguous minor 
	\[
	T\begin{pNiceMatrix}
		i     & i+1   & \Cdots & i+r-1 \\
		j     & j+1   & \Cdots & j+r-1
	\end{pNiceMatrix},
	\]
	with \(j > i\). 
	
	Now, we consider
	\[
	\Delta^{n-j+i} (T)= T\begin{pNiceMatrix}
		1         & \Cdots & n-j+i   \\
		j-i+1    & \Cdots & n
	\end{pNiceMatrix},
	\]
	and Fisher's inequality in Theorem \ref{teo:Classical} and Equation \eqref{eq:leq1} leads to
	\begin{align*}
	\Delta^{n-j+i}(T)& \leq T\begin{pNiceMatrix}
		1     & \Cdots & i+r-1 \\
		j-i+1 & \Cdots & j+r-1
	\end{pNiceMatrix}T\begin{pNiceMatrix}
		i+r   & \Cdots & n-j+i \\
		j+r   & \Cdots & n
	\end{pNiceMatrix} \\&\leq
	T\begin{pNiceMatrix}
		1     & \Cdots & i-1 \\
		j-i+1 & \Cdots & j-1
	\end{pNiceMatrix}T\begin{pNiceMatrix}
		i     & i+1   & \Cdots & i+r-1 \\
		j     & j+1   & \Cdots & j+r-1
	\end{pNiceMatrix}
	T\begin{pNiceMatrix}
		i+r   & \Cdots & n \\
		j+r   & \Cdots & n-j+i
	\end{pNiceMatrix}.
	\end{align*}
	As a result, when all nontrivial \(\Delta\)s are nonzero, all nontrivial contiguous minors above the main diagonal are also nonzero.
	
	The case \(j < i\) is treated analogously.  Theorem \ref{teo:Classical} and Equation \eqref{eq:leq2} gives
	\begin{align*}
	\Delta_{n-i+j} (T)&\coloneq T\begin{pNiceMatrix}
		i-j +1 & \Cdots & n \\
		1      & \Cdots & n-i+j
	\end{pNiceMatrix}\\&\leq T\begin{pNiceMatrix}
		i-j +1 & \Cdots & i+r-1 \\
		1      & \Cdots & j+r-1
	\end{pNiceMatrix}T\begin{pNiceMatrix}
		i     & i+1   & \Cdots & i+r-1 \\
		j     & j+1   & \Cdots & j+r-1
	\end{pNiceMatrix}T\begin{pNiceMatrix}
		i+r   & \Cdots & n \\
		j+r   & \Cdots & n-i+j
	\end{pNiceMatrix}.
	\end{align*}
	Hence, if all nontrivial \(\Delta\)'s are nonzero, then all nontrivial contiguous minors below the main diagonal are also nonzero.
	
	Therefore, we deduce that all nontrivial contiguous minors are also nonzero, and as a consequence, Theorem \ref{teo:BTP vs contiguous minors} yields the desired result.
\end{proof}

Now, let's outline a figure to illustrate the connection between a contiguous minor and its corresponding \(\Delta\):
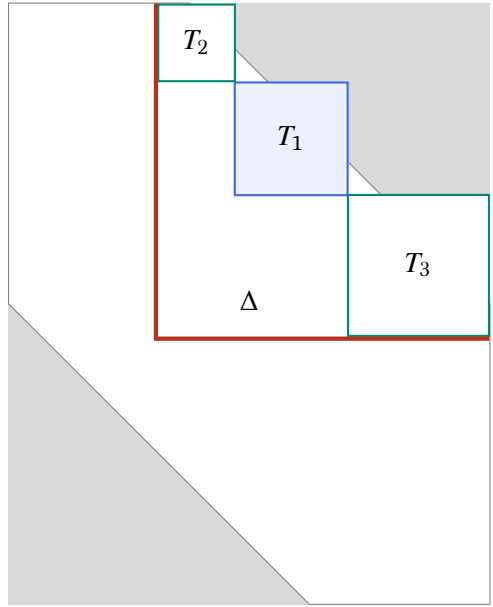
\begin{center}
	
	\drawmatrixset{bbox style={fill=Gray!30}}
	$\drawmatrix[bbox/.append, ,scale=8,banded, width=.8,
	]\Delta$
	
	\vspace*{-7cm}\hspace{1cm}
	$
	\drawmatrix[scale=1.5,draw=RoyalBlue,thick,fill=RoyalBlue!10]{T_1}$ 
	
	\vspace*{-2.6cm}
	\hspace{-1.4cm}$\drawmatrix[scale=1.5,draw=PineGreen,thick,height=.68,width=.68]{T_2}$
	
	\vspace*{1.45cm}
	\hspace{4.5cm}$\drawmatrix[scale=1.5,draw=PineGreen,thick,height=1.25,width=1.25]{T_3}$
	
	%
	%
	%
	%
	\vspace*{-4.5cm}
	\hspace*{-.67cm}	
	\begin{tikzpicture}
		\draw[ultra thick,BrickRed] (-4,6)--(-4,1.54)--(0.43,1.54);
		
	\end{tikzpicture}
	\hspace*{-2.62cm}	
\end{center}

\vspace*{3.5cm}

In the diagram provided, \(T_1\) denotes a nontrivial contiguous submatrix situated above the diagonal, while \(T_2\) represents the square block corresponding to the first application of Theorem \ref{teo:Classical}. Additionally, \(T_3\) signifies the diagonal block for the second application of Theorem \ref{teo:Classical}. Hence, it follows that \(\Delta \leq \text{det}(T_1) \cdot \text{det}(T_2) \cdot \text{det}(T_3)\). For \(j < i\), a straightforward reflection symmetry along the main diagonal in the depicted figure above provides a suitable scheme to describe the situation.

Given two banded matrices \(T_i\in B_{(p_i,q_i)}\), \(i\in\{1,2\}\), their products are also banded matrices with \(T_1T_2, T_2T_1\in  B_{(p_1+p_2,q_1+q_2)}\). Consequently, a pertinent question arises: what happens when we multiply two banded totally positive matrices? Is the product still a banded totally positive matrix? We will see in the next theorem that the Cauchy--Binet formula and Theorem \ref{teo:BTP vs contiguous minors} lead to a positive answer.

\begin{teo}\label{teo:semigroup}
	The set of banded totally positive matrices forms a semigroup.
\end{teo}
\begin{proof}
	According to Theorem \ref{teo:BTP vs contiguous minors}, it suffices to show that the nontrivial contiguous minors of the product \(T=T_1T_2\) of two banded totally positive matrices \(T_1\) and \(T_2\) do not vanish. For this purpose, we utilize the Cauchy--Binet theorem \cite{Fallat}. For the constituent minors, we have:
	\begin{align*}
		T\begin{pNiceMatrix}
			i&	i+1 &\Cdots & i+r-1\\
			j&	j+1&\Cdots & j+r-1
		\end{pNiceMatrix}=\sum_{\gamma} T_1\begin{pNiceMatrix}
			i&	i+1 &\Cdots & i+r-1\\
			\gamma_1&\gamma_2& \Cdots & \gamma_r
		\end{pNiceMatrix}T_2\begin{pNiceMatrix}
			\gamma_1 &\gamma_2&\Cdots & \gamma_r\\
			j&	j+1&\Cdots & j+r-1
		\end{pNiceMatrix}.
	\end{align*}
	It suffices to demonstrate that there exists at least one product in the right-hand side sum that is nonzero, implying that both factors are nonzero. For instance, when \(j=i-p_1-p_2\), we set
	\begin{align*}
		\begin{aligned}
			\gamma_k&=i+k-p_1-1, &  k&\in\{1,\dots, r\},
		\end{aligned}
	\end{align*}
	and all factors will not be zero by the assumption on \(T_1\) and \(T_2\). Similarly, for \(j=i+q_1+q_2\), we choose
	\begin{align*}
		\begin{aligned}
			\gamma_k&=i+k+q_1+1, &  k&\in\{1,\dots, r\},
		\end{aligned}
	\end{align*}
	and again all factors will not be zero by the assumptions on \(T_1\) and \(T_2\). Analogous arguments apply for all the minors in between these two, as well as for the other nontrivial contiguous minors. Hence, since all nontrivial contiguous minors are nonzero, the product is a banded totally positive matrix.
\end{proof}
		
\begin{lemma}
		For $(L)_{i,i-1}, (U)_{i-1,n}>0$, $i\in\{1,\dots,n\}$, the following bidiagonal matrices:
\begin{align}\label{eq:bidiagonal}
	L&=\begin{bNiceMatrix}
		1 & 0&\Cdots &&0\\
		(L)_{1,0}& 1&\Ddots&&\Vdots\\
		0&(L)_{2,1}&1 &&\\
		\Vdots&\Ddots&\Ddots[shorten-end=-8pt]&\Ddots&0\\
		0&\Cdots  &0 &(L)_{n-1,n-2}&1
	\end{bNiceMatrix}, &
		U&=\begin{bNiceMatrix}
		1 & (U)_{1,0}&0 &\Cdots&0\\
		0& 1&(U)_{2,1}&\Ddots&\Vdots\\
		\Vdots&\Ddots&1 &\Ddots[shorten-end=-20pt]&0\\
		\Vdots&&\Ddots[shorten-end=-8pt]&\Ddots&(U)_{n-2,n-1}\\
		0&\Cdots  &&0&1
	\end{bNiceMatrix}, 
\end{align}
are banded totally positive matrices.
\end{lemma} 
\begin{proof}
All contiguous row or column submatrices of triangular matrices are themselves triangular matrices. By the Gasca--Peña criterion, we can establish that these bidiagonal matrices belong to the class of InTN matrices. Consequently, the result follows immediately from Theorem \ref{teo:Deltas}.

For example, consider the lower triangular matrix \(L\). We observe that \(\Delta^n(L)=\Delta_{n}(L)=1\), and for \(i \in \{1,2,\dots,n-1\}\), we have:
\[
\Delta_{n-i}(L) = (L)_{1,0}  (L)_{2,1} \cdots   (L)_{n-i,n-i-1} > 0.
\]
These are the nontrivial \(\Delta\)s, and as they are all nonzero, we can conclude that \(L\) is banded totally positive. A similar argument can be applied to the upper triangular matrix \(U\).
\end{proof}

\begin{defi}
A matrix \(T\) is said to be PBF (positive of bidiagonal factorization) if it can be expressed as:
\begin{align}\label{eq:PBF}
	T=L_1\cdots L_p D U_q\cdots U_1
\end{align}
where \(L_i\) and \(U_i\) are positive bidiagonal matrices as defined in \eqref{eq:bidiagonal}, and \(D\) is a positive diagonal matrix.
\end{defi}

\begin{pro}\label{pro:PBF_BTP}
	PBF matrices are banded totally positive matrices.
\end{pro}

\begin{proof}
	Since all the factors in the PBF factorization are banded totally positive, their product is also banded totally positive.
\end{proof}

	\begin{defi}[Darboux transformations of banded matrices]
	Let us assume that $T$ admits a bidiagonal factorization \eqref{eq:PBF}). We consider a chain of new auxiliary matrices, called Darboux transformations, given by the contiguous permutation of the unitriangular matrices in the factorization \eqref{eq:PBF}
	\begin{align*}
			\hat T^{[+1]}&=L_2L_3   \cdots L_p  D U_q^{[N]} \cdots U_1  L_1 , \\
			\hat T^{[+2]}&=L_3  \cdots L_p     D   U_q  \cdots U_1 L_1 L_2  , \\ &\hspace*{5pt} \vdots\\
			\hat T^{[+(p-1)]}&=L_p     D   U_q  \cdots U_1 L_1 L_2   \cdots L_{p-1}  ,\\
			\hat T^{[+p]}&=   D   U_q  \cdots U_1 L_1 L_2   \cdots L_{p}  ,
	\end{align*}
	and 
	\begin{align*}
			\hat T^{[-1]}&=U_1 L_1  \cdots L_p     D   U_q  \cdots U_3U_2  , \\
			\hat T^{[-2]}&=U_2 U_1 L_1 \cdots L_p     D   U_q  \cdots U_3 , \\ &\hspace*{5pt} \vdots\\
			\hat T^{[-(q-1)]}&= U_{q-1}  \cdots U_1 L_1 L_2   \cdots L_{p}    D  U_{q}  ,\\
			\hat T^{[-q]}&= U_{q}  \cdots U_1 L_1 L_2   \cdots L_{p}    D   .
	\end{align*}
\end{defi}

\begin{pro}
	Darboux transformations of PBF matrices are banded totally positive matrices.
\end{pro}
\begin{proof}
	If follows from Theorem \ref{teo:semigroup} and Proposition \ref{pro:PBF_BTP}.
\end{proof}

\section{Normality  of  the recursion polynomials of banded totally positive matrices}

Let's introduce the recursion polynomials associated with a semi-infinite banded matrix $T\in$B$_{p,q}$:
\begin{defi}\label{def:Lambda-Upsilon}
	Given a semi-infinite matrix  $T\in$B$_{p,q}$, where each leading principal submatrix is in \(\text{B}_{(p,q)}\text{TP}\), we consider the corresponding left recursion matrix polynomial \( A(x)\in\R^{p\times \infty}[x] \), a \( p \times \infty \) matrix of polynomials, where \( A(x)T = xA(x) \), and a right recursion matrix polynomial \( B \in\R^{\infty\times q}[x] \), an \( \infty \times q \) matrix of polynomials, where \( TB(x) = xB(x) \). We arrange the matrix polynomial \( A \) by successive \( p \times p \) blocks \( A_{np} \), and the matrix \( B \) by successive \( q \times q \) blocks \( B_{nq} \). 
\end{defi}

The entries in these recursion matrices are polynomials:
\begin{align*}
	A&=\left[\begin{NiceMatrix}
		A_0 &A_p& A_{2p}&\Cdots
	\end{NiceMatrix}\right], &A_{Np}&=\begin{bNiceMatrix}
	A_{Np}^{(1)}&\Cdots&A^{(1)}_{(N+1)p-1}\\
	\Vdots &&\Vdots\\
		A_{Np}^{(p)}&\Cdots&A^{(p)}_{(N+1)p-1}
	\end{bNiceMatrix},\\
	B&=\left[\begin{NiceMatrix}
		B_0 \\B_q\\ B_{2q}\\\Vdots
	\end{NiceMatrix}\right], &B_{Nq}&=\begin{bNiceMatrix}
		B_{Nq}^{(1)}&\Cdots&B^{(q)}_{Nq}\\
		\Vdots &&\Vdots\\
		B_{(N+1)q-1}^{(1)}&\Cdots&B^{(q)}_{(N+1)q-1}
	\end{bNiceMatrix}.
\end{align*}

The entries of these left and right eigenvectors are polynomials in the eigenvalue $x$, known as left and right recursion polynomials, respectively, determined by the initial conditions
\begin{align}
	\label{eq:initcondtypeI}
	\begin{cases}
		A^{(1)}_0=1 , \\
		A^{(1)}_1= \nu^{(1)}_1 , \\
		\hspace{.895cm} \vdots \\
		A^{(1)}_{p-1}=\nu^{(1)}_{p-1} ,
	\end{cases}
	&&
	\begin{cases}
		A^{(2)}_0=0 , \\
		A^{(2)}_1= 1 , \\
		A^{(2)}_2= \nu^{(2)}_2 , \\
		\hspace{.895cm} \vdots \\
		A^{(2)}_{p-1}=\nu^{(2)}_{p-1} ,
	\end{cases}
	&& \cdots &&
	\begin{cases}
		A^{(p)}_0 =0 , \\
		\hspace{.915cm} \vdots \\
		A^{(p)}_{p-2} = 0 , \\
		A^{(p)}_{p-1} = 1,
	\end{cases}\\
	\label{eq:initcondtypeII}
	\begin{cases}
		B^{(1)}_0=1 , \\
		B^{(1)}_1= \xi^{(1)}_1 , \\
		\hspace{.895cm} \vdots \\
		B^{(1)}_{q-1}=\xi^{(1)}_{q-1} ,
	\end{cases}
	&&
	\begin{cases}
		B^{(2)}_0=0 , \\
		B^{(2)}_1= 1 , \\
		B^{(2)}_2= \xi^{(2)}_2 , \\
		\hspace{.895cm} \vdots \\
		B^{(2)}_{q-1}=\xi^{(2)}_{q-1} ,
	\end{cases}
	&& \cdots &&
	\begin{cases}
		B^{(q)}_0 =0 , \\
		\hspace{.915cm} \vdots \\
		B^{(q)}_{q-2} = 0 , \\
		B^{(q)}_{q-1} = 1,
	\end{cases}
\end{align}
with $\nu^{(i)}_{j},\xi^{(i)}_{j}$ being arbitrary constants. 

The initial conditions can be written in matrix form:
\begin{align}
	\label{eq:ic}
	A_0&\coloneq \begin{bNiceMatrix}
		1& 	\nu^{(1)}_1 & \Cdots& && \nu^{(1)}_{p-1} \\
	0 & 1 & \Ddots&& & \Vdots \\
		\Vdots & \Ddots[shorten-start=5pt,shorten-end=0pt] & \Ddots[shorten-end=3pt]& && \\
		&&&& &\\&&&&& \nu^{(p-1)}_{p-1}\\[10pt]
		0&\Cdots& &&0& 1 
	\end{bNiceMatrix} ,&
	B_0&\coloneq \begin{bNiceMatrix}
		1& 0 & \Cdots& && 0 \\
		\xi^{(1)}_1 & 1 & \Ddots&& & \Vdots \\
		\Vdots & \Ddots[shorten-start=-3pt,shorten-end=-3pt] & \Ddots& && \\
		&&&& &\\&&&&&0\\
		\xi^{(1)}_{q-1} &\Cdots& && \xi^{(q-1)}_{q-1}& 1 
	\end{bNiceMatrix}.
\end{align}

We use the ceiling function $\lceil x\rceil$ that maps $x$ to the least integer greater than or equal to $x$.
\begin{pro}\label{pro:the_correct_one}
Let's assume that the initial conditions $A_0$ and $B_0$ are nonsingular matrices. 	Then, the degrees of the recursion polynomials of $T\in$B$_{p,q}$, with nonzero extreme subdiagonal and superdiagonal,  are bounded as follows  
\[
\begin{aligned}
		\operatorname{deg} A_n^{(a)}&\leq\left\lceil\frac{n+2-a}{p}\right\rceil-1, & \operatorname{deg} B_n^{(b)}\leq\left\lceil\frac{n+2-b}{q}\right\rceil-1.
\end{aligned}
\]
Equalities are ensured for $a\in\{1,\dots,p\}$ with  $n=Np+a-1$,  for left polynomials and for $b\in\{1,\dots,q\}$ with $n=Nq+b-1$, for right polynomials.

\end{pro}
\begin{proof}
We will present the proof for the left polynomials; a similar approach can be followed for the right polynomials.
	
To analyze the recursion polynomials within \( A \), we observe that it is built up with \( p \times p \) blocks. As a block matrix, the banded matrix $T$ cane be represented in terms of its blocks $\Theta_{k,l}\in\R^{p\times p}$, $k,l\in\N_0$.  As a block matrix there is just one  subdiagonal with blocks \( \Theta_{n,n-1} \) being \( p \times p \) nonsingular upper triangular matrices. The number of block's superdiagonals will be \( N = \lceil \frac{q}{p} \rceil \). 
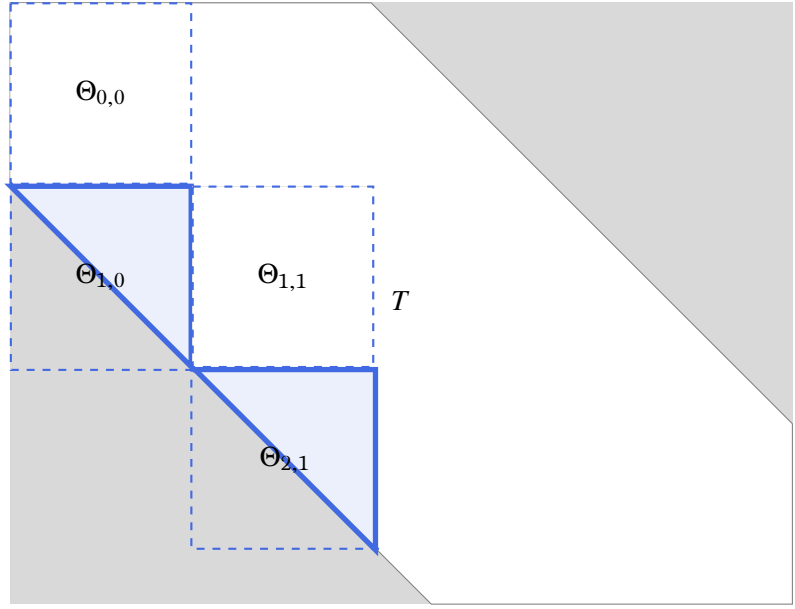
\begin{center}
	
	\drawmatrixset{bbox style={fill=Gray!30}}
	$\drawmatrix[bbox/.append, ,scale=8,banded, width=1.3,
	]T$
	
	\vspace*{-8.05cm}	\hspace{-8.1cm}
	$
	\drawmatrix[scale=2.4,draw=RoyalBlue,thick,dashed
	]{\Theta_{0,0}}$ 
	
		\vspace*{-.05cm}\hspace{-8.1cm}
	$
	\drawmatrix[upper,bbox/.append,scale=2.4,draw=RoyalBlue,line width=2pt,fill=RoyalBlue!10
	]{\Theta_{1,0}}$ 
	
		\vspace*{-2.48cm}\hspace{-3.26cm}
		$
	\drawmatrix[scale=2.4,draw=RoyalBlue,thick,dashed,
	]{\Theta_{1,1}}$ 
	
				\vspace*{-.05cm}\hspace{-3.2cm}
	$
	\drawmatrix[upper,bbox/.append,scale=2.4,draw=RoyalBlue,,line width=2pt,fill=RoyalBlue!10
	]{\Theta_{2,1}}$ 
%
%
	
	%
	%
	%
	%
	\vspace*{-4.8cm}
	\hspace*{-5.48cm}	
	\begin{tikzpicture}
		\draw[ thick,dashed,RoyalBlue] (-4,6)--(-4,3.68)--(-1.6,3.68)--(-1.6,1.3)--(0.9,1.3);
	\end{tikzpicture}
		\vspace*{1cm}
\end{center}

The recursion relations are given by:
\begin{align*}
	A_0 \Theta_{0,0}+A_{p}\Theta_{1,0}&=x A_0,\\
	A_0 \Theta_{0,1}+A_{p}\Theta_{1,1}+A_{2p}\Theta_{2,1}&=x A_p,\\
	&\hspace{5pt}\vdots\\
		A_0 \Theta_{0,N-1}+A_{p}\Theta_{1,N-1}+\cdots+A_{Np}\Theta_{N,N-1}&=x A_{(N-1)p},\\
		A_{p}\Theta_{1,N}+A_{2p}\Theta_{2,N}+\cdots+A_{Np}\Theta_{N+1,N}&=x A_{Np},\\
		&\hspace{5pt}\vdots
\end{align*}
Since the lower extreme diagonal has no zero entries, all these upper triangular blocks are nonsingular and have an inverse matrix. Therefore,
	\begin{align*}
		A_p(x)&= x A_{0}\Theta_{1,0}^{-1}+O(1),\\
		A_{2p}(x)&=x^2 A_{0}\Theta_{1,0}^{-1}\Theta_{2,1}^{-1}+O(x),\\
		&\hspace{5pt}\vdots\\
		A_{Np}(x)&=x^N A_0\Theta_{1,0}^{-1}\Theta_{2,1}^{-1}\cdots \Theta_{N,N-1}^{-1}+O(x^{N-1}),\\
		A_{(N+1)p}(x)&=x^{N+1} A_0\Theta_{1,0}^{-1}\Theta_{2,1}^{-1}\cdots \Theta_{N+1,N}^{-1}+O(x^{N}),\\
		&\hspace{5pt}\vdots
	\end{align*}
Therefore, the leading coefficient $A_{kp,k}$ for the block $A_{kp}$ 
\begin{align*}
	A_{kp}&=A_{kp,k}x^k+O(x^{k-1}), & k&\in\N_0
\end{align*}
is the upper triangular matrix.
	\begin{align*}
A_{kp,k}=	A_0\Theta_{1,0}^{-1}\Theta_{2,1}^{-1}\cdots \Theta_{k,k-1}^{-1}.
	\end{align*}
	Assuming $A_0$ is not singular, we have that
	\begin{align*}
		A_{kp,k}=(\Theta_{k,k-1}\cdots \Theta_{1,0}A_0^{-1})^{-1}.
	\end{align*}
	Therefore,  since the set of nonsingular upper triangular matrices is a group, this leading matrix block is a nonsingular upper triangular matrix. Consequently, from the diagonal of the leading coefficient we get
	\begin{align*}
		\deg A^{(1)}_{Np}&= N,\\
		\deg A^{(2)}_{Np+1}&= N,\\
			\deg A^{(3)}_{Np+2}&= N,\\
			&\hspace{4pt}\vdots\\
			\deg A^{(p)}_{(N+1)p-1}&= N,
	\end{align*}
	and from the off diagonal entries that
\[	\begin{aligned}
	\deg A^{(1)}_{Np+a}&\leq N, &a&\in\{1,\dots,p-1\},\\
	\deg A^{(2)}_{Np+a}&\leq N, &a&\in\{2,\dots,p-1\},\\
	\deg A^{(3)}_{Np+a}&\leq N,&a&\in\{3,\dots,p-1\},\\
	&\hspace{4pt}\vdots
	\end{aligned}
	\]
Consequently, we derive $\operatorname{deg} A_n^{(a)}\leq\left\lceil\frac{n+2-a}{p}\right\rceil-1$, with equality attained at least when $n=Np+a-1$ for $a\in\{1,\dots,p\}$, pertaining to left polynomials, and when $n=Nq+b-1$ for $b\in\{1,\dots,q\}$, relating to right polynomials.
	\end{proof}
	
	\begin{defi}[Step line  normality]
We say the triple $(T,A_0,B_0)$ is  normal (or step line  normal), whenever the corresponding recursion polynomials satisfy:
		\[
		\begin{aligned}
			\operatorname{deg} A_n^{(a)}&=\left\lceil\frac{n+2-a}{p}\right\rceil-1, & \operatorname{deg} B_n^{(b)}=\left\lceil\frac{n+2-b}{q}\right\rceil-1.
		\end{aligned}
		\]
		Here $a\in\{1,\dots,p\}$, $b\in\{1,\dots,q\}$ and $n\in\N_0$.
	\end{defi}
	
It's essential to highlight that within the realm of generic banded matrices, the presumption of normality across all indices along the step line  isn't automatic. As Proposition \ref{pro:the_correct_one} illuminates, the assurance of step line  normality for indices isn't guaranteed universally. Instead, it's only established under specific conditions. For left polynomials, normality is achieved when $a$ falls within $\{1,\dots,p\}$ and $n=Np+a-1$, while for right polynomials, normality prevails when $b$ falls within $\{1,\dots,q\}$ with $n=Nq+b-1$.

\begin{coro}
Normality can be characterized by ensuring that the upper triangular matrix $A_{kp,k}$, denoting the leading block coefficient of $A_{kp}(x)$, has all its nontrivial entries—those on and above the main diagonal—distinct from zero. Similarly, for the lower triangular matrix $B_{kq,k}$, representing the leading block coefficient of $B_{kq}(x)$,  normality demands that all its nontrivial entries—those on and below the main diagonal—are distinct from zero.
\end{coro}

We will now demonstrate that the assurance of  normality  is achieved within the framework of banded total positivity.

\begin{teo}[Normality and total positivity]\label{teo:normality in the step line }
	 Assume that  $A_0^{-1}\in$\scalebox{.7}{\UT}  \hspace{-6pt}TP, $B_0^{-1}\in$\scalebox{.65}{\LT}  \hspace{-6pt}TP and  $T\in$BTP. Then, $(T,A_0,B_0)$ is  normal.
\end{teo}
	
\begin{proof}
	We outline the proof for the left polynomials; analogous steps can be taken for the right polynomials.
	
	Since all the submatrices $\Theta_{l,l-1}$ are upper triangular and totally positive, and assuming $A_0^{-1}$ is also upper triangular  totally positive, by virtue of Theorem \ref{teo:semigroup}, we deduce that $A_{kp,k}$ is the inverse of an upper triangular totally positive matrix $M=\Theta_{k,k-1}\cdots \Theta_{1,0}A_0^{-1}$. However, the inverse of a matrix is constructed using signed minors in the adjugate matrix, i.e.
	\begin{align*}
(-1)^{i+j}	M\begin{bNiceMatrix}
		1 &\Cdots& i-1&i+1&\Cdots& p\\
		1 &\Cdots& j-1&j+1&\Cdots&p
	\end{bNiceMatrix}.
	\end{align*}
with $i\geq j$, otherwise are trivial,	but these are non trivial minors because 
\begin{align*}
	\{1 ,\dots, i-1,i+1,\dots, p\}\leq
\{1 ,\dots, j-1,j+1,\dots, p\}.
\end{align*}
Therefore, all the nontrivial entries  in the leading coefficient $A_{kp,k}=M^{-1}$ of $A_{kp}(x)$ are not zero, and all degrees are achieved, and $T$ is step line  normal.
\end{proof}
\begin{coro} Suppose the condition outlined in Theorem \ref{teo:normality in the step line } is satisfied. In this case, the leading block coefficient $A_{kp,k}$ (or $B_{kq,k}$) of the block recursion polynomial $A_{kn}(x)$ (or $B_{kq}(x)$) manifests as an upper (or lower) triangular matrix with a distinctive checkerboard pattern. i.e., $(-1)^j$ times the $j$-th superdiagonal (subdiagonal) entries are positive.
	\end{coro}

\begin{rem}
In reference to \cite{BFM1}, it was inaccurately asserted that normality holds within the step line . The accurate assertion, aligned with the assumptions outlined in that paper, is presented in Proposition \ref{pro:the_correct_one}. Importantly, this erroneous claim does not impact any other findings within the paper. Notably, spectral results, such as the Favard theorem, and mixed multiple Gauss quadrature remain valid even in scenarios where normality does not hold.

Furthermore, by considering initial conditions within $\triangle$TP, step line  normality is indeed attained, as elucidated in the preceding theorem, thereby resolving the issue.
\end{rem}

We will now delve into the spectral Favard theorem as delineated in \cite{BFM1}, and propose a modification to incorporate step line  normality. To achieve this, we first introduce two crucial matrices, based on a positive bidiagonal factorization as described in \eqref{eq:bidiagonal}:

\begin{defi}
	Associated with the PBF factorization \eqref{eq:PBF}, we define the \((p+1)\times (p+1)\) matrix \(\Lambda_p\) as:
	\begin{align*}
		\Lambda_p\coloneq\begin{bNiceMatrix}
			1 & 1 &1 &\Cdots &1& 1\\
			0& (L_1)_{1,0}&(L_1L_2)_{1,0}& \Cdots&(L_1\cdots L_{p-1})_{1,0}&(L_1\cdots L_p)_{1,0}\\
			0& 0&(L_1L_2)_{2,0}& \Cdots&(L_1\cdots L_{p-1})_{2,0}&(L_1\cdots L_p)_{2,0}\\
			\Vdots&&\Ddots
			&\Ddots&&\Vdots\\\\
			0&\Cdots &&&0&(L_1\cdots L_p)_{p,0}
		\end{bNiceMatrix}
	\end{align*}
	and  the \((q+1)\times (q+1)\) matrix \(\Upsilon_q\) as:
	\begin{align*}
		\Upsilon_q\coloneq\begin{bNiceMatrix}
			1 & 0 &0 &\Cdots && 0\\
			1& (U_1)_{0,1}&0& &&\Vdots\\
			1& (U_2U_1)_{0,1}&(U_2U_1)_{0,2}& \Ddots&&\\
			\Vdots&\Vdots&\Vdots&\Ddots[shorten-end=-10pt]&&\\
			1&(U_{p-1}\cdots U_1)_{0,1} &(U_{p-1}\cdots U_1)_{0,2}&&&0\\
			1&(U_p\cdots U_1)_{0,1}&(U_p\cdots U_1)_{0,2}&\Cdots&&(U_p\cdots U_1)_{0,p}
		\end{bNiceMatrix}
	\end{align*}
\end{defi}

\begin{pro}\label{pro:Lambda}
	The matrices \(\Lambda_p\) and \(\Upsilon_q\) belong to \scalebox{.7}{\UT}  \hspace{-6pt}TP and \scalebox{.65}{\LT}  \hspace{-6pt}TP, respectively.
\end{pro}
\begin{proof}
	We give the proof for $\Lambda_p$, for $\Upsilon_q$ one proceeds analogously.
	
	For the base case \(p=1\), the statement is trivially true. Now, we proceed with the induction step, assuming that it holds for \(\Lambda_{p-1}\) and showing that it holds for \(\Lambda_p\).
	
	We will analyze various families of \(p\times p\) submatrices of \(\Lambda_p\), encompassing all possible submatrices of \(\Lambda_p\), and demonstrate that they are banded totally positive, meaning they have no zero nontrivial minors. Consequently, all nontrivial minors will be nonzero.
	\begin{enumerate}
		\item For the first submatrix we remove the first column and last row 
		\begin{align*}
			M=\begin{bNiceMatrix}
				1 &1 &\Cdots &1& 1\\
				(L_1)_{1,0}&(L_1L_2)_{1,0}& \Cdots&(L_1\cdots L_{p-1})_{1,0}&(L_1\cdots L_p)_{1,0}\\
				0&(L_1L_2)_{2,0}& \Cdots&(L_1\cdots L_{p-1})_{2,0}&(L_1\cdots L_p)_{2,0}\\
				\Vdots&\Ddots
				&\Ddots[shorten-end=-35pt]&\Vdots&\Vdots\\
				0&\Cdots &0&(L_1\cdots L_{p-1})_{p-1,0}&(L_1\cdots L_p)_{p-1,0}
			\end{bNiceMatrix}\in \R^{p\times p}.
		\end{align*}
		This submatrix factors as follows
		\begin{align*}
			M=\begin{bNiceMatrix}
				1 & 0&\Cdots &&0\\
				(L_1)_{1,0}& 1&\Ddots&&\Vdots\\
				0&(L_1)_{2,1}&1 &&\\
				\Vdots&\Ddots&\Ddots[shorten-end=-10pt]&\Ddots&0\\
				0&\Cdots  &0 &(L_1)_{p-1,p-2}&1
			\end{bNiceMatrix}\begin{bNiceMatrix}
				1 & 1 &1 &\Cdots &1& 1\\
				0& (L_2)_{1,0}&(L_2L_3)_{1,0}& \Cdots&(L_2\cdots L_{p-1})_{1,0}&(L_2\cdots L_p)_{1,0}\\
				0& 0&(L_2L_3)_{2,0}& \Cdots&(L_2\cdots L_{p-1})_{2,0}&(L_2\cdots L_p)_{2,0}\\
				\Vdots&&\Ddots
				&\Ddots[shorten-end=-2pt]&&\Vdots\\\\
				0&\Cdots &&&0&(L_2\cdots L_p)_{p-1,0}
			\end{bNiceMatrix}.
		\end{align*}
		The first factor is a positive bidiagonal matrix, hence it is banded totally positive. The second factor, by the induction hypothesis, is also banded totally positive. Therefore, their product, according to Theorem \ref{teo:semigroup}, is banded totally positive. Consequently, all nontrivial minors encompassed within \(M\) are nonzero.
		
		However, not all nontrivial minors of \(\Lambda_p\) lie along \(M\). Therefore, in the following items, we will consider the remaining minors.
		\item The first family of remaining minors consists of minors of the $p\times p $ submatrices \(R_l\), \(l\in\{1,\dots,p\}\), that appear by removing the $(l+1)$-th column and last row of $\Lambda_p$, or equivalently
		 obtained from \(M\) by removing the \(l\)-th column and adding as the first column $\begin{bNiceMatrix}[small]
			1\\0\\\Vdots\\0
		\end{bNiceMatrix}$.  Let us study submatrices of these submatrices
		\begin{align*}
			R_l\begin{bNiceMatrix}
				i_1 &\Cdots & i_N\\
				j_1&\Cdots & j_N
			\end{bNiceMatrix}.
		\end{align*}
	with \(j_1 = 0\); otherwise, these submatrices will be submatrices of \(M\). Moreover, to ensure nontrivial minors, we must have \(i_1 = 1\), as otherwise, the corresponding submatrix will contain a zero column. Thus, we observe that nontrivial submatrices take the form:
		\begin{align*}
			R_l\begin{bNiceMatrix}
				1&	i_2 &\Cdots & i_k\\
				1&	j_2&\Cdots & j_k
			\end{bNiceMatrix}=\left[\begin{NiceArray}{c|ccc}
				1&	* &\Cdots &*\\\hline
				0&	\Block{3-3}<\Large>{M_l}& &\\
				\Vdots &&&\\
				0&&&
			\end{NiceArray}\right].
		\end{align*}
		where $M_l$ is a nontrivial submatrix of $M$, and consequently, these nontrivial submatrices are nonsingular.
		\item For the second family we take submatrices $S_l$, $l\in\{1,\dots, p\}$, that appear by removing the first column and $(l-1)$-th row of $\Lambda_p$, or equivalently 
		obtained from $M$ by removing the $l$-th row and adding as last row $\begin{bNiceMatrix}[small]
			0&\Cdots& 0 & (L_1\cdots L_p)_{p,0}
		\end{bNiceMatrix}$. Analogous arguments as in the previous  item leads to look for nontrivial submatrices only among the submatrices
		\begin{align*}
			S_l\begin{bNiceMatrix}
				i_1&\Cdots& i_{k-1} & p\\
				j_1&\Cdots &j_{k-1} & p
			\end{bNiceMatrix}=\left[\begin{NiceArray}{ccc|c}
				\Block{3-3}<\Large>{N_l}& &&*\\
				&&&\Vdots\\
				&&&*\\\hline
				0&\Cdots&0&(L_1\cdots L_p)_{p,0}
			\end{NiceArray}\right].
		\end{align*}
		where $N_l$ is a nontrivial submatrix of $M$. Hence, these are nonsingular nontrivial submatrices.
		\item Finally, we consider submatrices \(T_{l,k}\), \(k,l\in\{1,\dots,p\}\), obtained from \(M\) by removing the \(l\)-th column and the \(k\)-th row, and adding as the first column: $\begin{bNiceMatrix}[small]
			1\\0\\\Vdots\\0
		\end{bNiceMatrix}$ and last row $\begin{bNiceMatrix}[small]
			0&\Cdots& 0&
 (L_1\cdots L_p)_{p,0}
		\end{bNiceMatrix}$. As before, the nontrivial submatrices already not considered are among
		\begin{align*}
			T_{l,k} \begin{bNiceMatrix}1&	i_2&\Cdots& i_{k-1} & p\\
				1&j_2&\Cdots &j_{k-1} & p
			\end{bNiceMatrix}=\left[\begin{NiceArray}{c|ccc|c}
				1 &*&*&*&*\\\hline
				0&  \Block{3-3}<\Large>{M_{l,k}}&&&*\\
				\Vdots&&&&\Vdots\\
				0&&&&*\\\hline
				0&0&\Cdots&0&(L_1\cdots L_p)_{p,0}
			\end{NiceArray}\right]
		\end{align*}
		where $M_{l,k}$ is a nontrivial submatrix of $M$. Therefore, these nontrivial submatrices are not singular. 
	\end{enumerate}
	Therefore, we conclude that all nontrivial minors of \(\Lambda_p\) are nonsingular.
\end{proof}

In \cite[Theorem 9.1]{BFM1} we proved a spectral Favard theorem, we reword it as follows: 
\begin{teo}[Favard spectral representation]\label{theorem:spectral_representation_bis}
	Let us assume that
	\begin{enumerate}
		\item The banded  matrix $T$ is bounded and there exist $s\geqslant 0$ such that $T+sI$ has a PBF \eqref{eq:PBF}.
		\item The sequences $\big\{A^{(1)}_n,\dots,A^{(p)}_n\big\}_{n=0}^\infty, \big\{B^{(1)}_n,\dots,B^{(q)}_n\big\}_{n=0}^\infty$ 
		of recursion polynomials are determined by the initial condition matrices $A_0$ and $B_0$,respectively, such that $A_0^{-1}=\Lambda_p\mathcal A$, $B_0^{-1}=\mathcal B \Upsilon_q$, and $\mathcal A\in \R^{p\times p}$ is a nonnegative upper unitriangular matrices and $\mathcal B\in \R^{q\times q}$ is a nonnegative lower unitriangular matrix. Here the matrices $\Lambda_p$ and $\Upsilon_q$ are described in Definition \ref{def:Lambda-Upsilon}
	\end{enumerate}
	Then, there exists $pq$ non decreasing positive functions $\psi_{b,a}$, $a\in\{1,\dots,p\}$ and $b\in\{1,\dots,q\}$
	and corresponding positive Lebesgue--Stieltjes measures $\d\psi_{b,a}$ with compact support $\Delta$ such that the following mixed multiple biorthogonality holds
	\begin{align*}
		\sum_{a=1}^p\sum_{b=1}^q	\int_{\Delta} B^{(b)}_l(x)\d\psi_{b,a}(x) A^{(a)}_{k}(x)&= \delta_{k,l}, &k,l&\in\N_0.
	\end{align*}
\end{teo}
The conditions delineated in ii) guarantee the positivity of all Christoffel numbers, thus ensuring the existence of measures. However, these conditions alone do not guarantee normality within the step line .

To establish normality in the spectral Favard theorem for mixed multiple orthogonality, it is imperative to constrain the degrees of freedom pertaining to the matrices $\mathcal A$ and $\mathcal B$, transitioning from mere positivity to banded total positivity. With this adjustment, the revised Favard theorem reads as follows:
\begin{teo}[Favard spectral representation with normality on the step line ]\label{theorem:spectral_representation_tris}
	Let us assume that
	\begin{enumerate}
	\item The banded  matrix $T$ is bounded and there exist $s\geqslant 0$ such that $T+sI$ has a PBF \eqref{eq:PBF}.
	\item The sequences $\big\{A^{(1)}_n,\dots,A^{(p)}_n\big\}_{n=0}^\infty, \big\{B^{(1)}_n,\dots,B^{(q)}_n\big\}_{n=0}^\infty$ 
	of recursion polynomials are determined by the initial condition matrices $A_0$ and $B_0$,respectively, such that $A_0^{-1}=\Lambda_p\mathcal A$, $B_0^{-1}=\mathcal B \Upsilon_q$, and $\mathcal A\in \R^{p\times p}$ is a  upper unitriangular totally positive matrix and $\mathcal B\in \R^{q\times q}$ is a  lower unitriangular totally positive matrix. Here the matrices $\Lambda_p$ and $\Upsilon_q$ are described in Definition \ref{def:Lambda-Upsilon}.
\end{enumerate}
Then, there exists $pq$ non decreasing positive functions $\psi_{b,a}$, $a\in\{1,\dots,p\}$ and $b\in\{1,\dots,q\}$
and corresponding positive Lebesgue--Stieltjes measures $\d\psi_{b,a}$ with compact support $\Delta$ such that the following mixed multiple biorthogonality holds
\begin{align*}
	\sum_{a=1}^p\sum_{b=1}^q	\int_{\Delta} B^{(b)}_l(x)\d\psi_{b,a}(x) A^{(a)}_{k}(x)&= \delta_{k,l}, &k,l&\in\N_0,
\end{align*}
and the degrees of these mixed multiple orthogonal polynomials are
	\[
\begin{aligned}
	\deg A_n^{(a)}&=\left\lceil\frac{n+2-a}{p}\right\rceil-1, & \deg B_n^{(b)}=\left\lceil\frac{n+2-b}{q}\right\rceil-1.
\end{aligned}
\]

\end{teo}

\begin{proof}
It follows from the proof of  Theorem \ref{theorem:spectral_representation_bis} in \cite{BFM1}, Proposition \ref{pro:Lambda}, Theorems \ref{teo:semigroup} and  \ref{teo:normality in the step line }.
\end{proof}

\section*{Acknowledgments}

AB acknowledges Centre for Mathematics of the University of Coimbra 
(funded by the Portuguese Government through FCT/MCTES, \href{https://doi.org/10.54499/UIDB/00324/2020}{doi:10.54499/UIDB/00324/2020}.

AF acknowledges the CIDMA Center for Research and Development in Mathematics and Applications
(University of Aveiro) and the Portuguese Foundation for Science and Technology (FCT) for their support within
projects, 
\href{https://doi.org/10.54499/UIDB/04106/2020 }{doi:10.54499/UIDB/04106/2020} and   \href{https://doi.org/10.54499/UIDP/04106/2020 }{doi:10.54499/UIDP/04106/2020}.

MM acknowledges Spanish ``Agencia Estatal de Investigación'' research project [PID2021- 122154NB-I00], \emph{Ortogonalidad y Aproximación con Aplicaciones en Machine Learning y Teoría de la Probabilidad}.

	 
	\end{document}